\documentclass[reqno,twoside,12pt]{amsart}

\hoffset - 2.2cm

\usepackage{amstext}
\usepackage{pstricks,pst-node,pst-text,pst-3d}
\usepackage{color}
\usepackage{amsmath}
\usepackage{amsthm}
\usepackage{amssymb}

\usepackage{enumerate}

\theoremstyle{definition}
\newtheorem{Definition}{Definition}[section]
\theoremstyle{plain}
\newtheorem{Theorem}{Theorem}[section]
\newtheorem{Lemma}[Theorem]{Lemma}
\newtheorem{Remark}[Theorem]{Remark}
\newtheorem{Corollary}[Theorem]{Corollary}

\newtheorem{Example}{Example}[section]

\newcommand{\bH}{\mathbb{H}}
\newcommand{\bI}{\mathbb{I}}

\newcommand{\bN}{\mathbb{N}}

\newcommand{\bR}{\mathbb{R}}

\newcommand{\bV}{\mathbb{V}}

\newcommand{\bZ}{\mathbb{Z}}

\newcommand{\cH}{\mathcal{H}}

\newcommand{\cP}{\mathcal{P}}

\newcommand{\I}{\mathbb{I}}

\newcommand{\Id}{\operatorname{Id}}
\newcommand{\diag}{\operatorname{diag}}

\newcommand{\SO}{\operatorname{SO}}

\renewcommand{\H}{\mathbb{H}}

\renewcommand{\span}{\operatorname{span}}

\newcommand{\BIF}{\mathcal{BIF}}
\newcommand{\LB}{\Delta_{S^{n-1}}}
\numberwithin{equation}{section}

\begin{document}

\title[Elliptic systems on geodesic balls]{Rabinowitz alternative for non-cooperative \\ elliptic systems on geodesic balls}

\author{S{\l}awomir Rybicki}
\address{Faculty of Mathematics and Computer Science\\
Nicolaus Copernicus University \\
PL-87-100 Toru\'{n} \\ ul. Chopina $12 \slash 18$ \\
Poland}

\email{rybicki@mat.umk.pl}

\author{Naoki Shioji}
\address{Department of Mathematics \\ Faculty of Engineering\\ Yokohama National University \\ Tokiwadai, Hodogaya-ku \\ Yokohama 240-8501, Japan}

\email{shioji@ynu.ac.jp}

\author{Piotr Stefaniak}
\address{School of Mathematics\\
West Pomeranian University of Technology, PL-70-310 Szczecin \\ al. Piast\'{o}w $48\slash 49$ \\ Poland}

\email{pstefaniak@zut.edu.pl}
\date{\today}

\thanks{Partially supported by the National Science Center,  Poland,  under grant    DEC-2012/05/B/ST1/02165}

\begin{abstract}
The purpose  of this paper is to study properties of continua (closed connected sets) of  nontrivial solutions of  non-cooperative elliptic systems considered on  geodesic balls in $S^n$. In particular, we have shown that if the geodesic ball is a hemisphere, then these continua are unbounded. It is also shown that the  phenomenon of global symmetry-breaking bifurcation of such solutions occurs. Since the problem is variational and $\SO(n)$-symmetric, we apply the techniques of equivariant bifurcation theory  to prove the main results of this article. As the  topological tool we  use the degree theory for $\SO(n)$-invariant strongly indefinite functionals defined in  \cite{degree}.
\end{abstract}

\maketitle

\section{Introduction}

The aim of this paper is to study continua  of solutions of boundary value  problems for non-cooperative elliptic systems considered on  geodesic balls in $S^n$, i.e. systems of the form
\begin{equation} \label{Dirichletwstep}
\left\{
\begin{array}{rclll}
\Lambda \Delta_{S^{n}} u & = &\nabla_u F(u,\lambda) & \text{ in } & B(\gamma), \\ u & = & 0  & \text{ on } & \partial B(\gamma),
\end{array}
\right.
\end{equation}
where $\Lambda= \diag(\alpha_1,\ldots,\alpha_p), \alpha_i=\pm 1,$
$\Delta_{S^n}$ is the Laplace-Beltrami operator on the sphere $S^n, B(\gamma)\subset S^n$ is a geodesic ball of radius $\gamma > 0,$
 $F\in C^2(\bR^p\times\bR,\bR)$ is such that $\nabla_u F(u,\lambda)=\lambda u+\nabla_u \eta (u,\lambda)$, where $\eta \in C^2(\bR^p\times\bR,\bR)$, $\nabla_u \eta (0,\lambda)=0$ and $\nabla^2_u \eta(0,\lambda)=0$ for every $\lambda\in\bR$.

More precisely speaking, our purpose is to study the phenomenon of global bifurcations of weak solutions of system \eqref{Dirichletwstep}. In other words we have studied closed connected sets of weak solutions of this system, satisfying a Symmetric Rabinowitz alternative. For the classical Rabinowitz alternative see for instance \cite{RB,JI,Nierenberg,Rabinowitz,Rabinowitz1}.

Global bifurcations of solutions of nonlinear problems have been studied  under various conditions by many authors. Some  references and discussion of these results can be found in \cite{LB}. Here we discuss only some results concerning symmetric nonlinear problems, where the authors have eliminated one of the Rabinowitz alternatives showing that some (or all)  global solution branches  are bounded (or unbounded).

An elliptic boundary value problem on a two-dimensional annulus has been considered  in \cite{Dancer1}.  The author has studied the bifurcation of non-radially symmetric  solutions from radially symmetric positive ones.  The existence of many distinct global branches of non-symmetric solutions which do not intersect has been shown in this article.

 The authors of  \cite{HeaKie0} have studied global symmetry-breaking equilibria of the van der Waals-Cahn-Hilliard phase-field model on the sphere $S^2.$ What is interesting they have proved that all the continua of nontrivial solutions are bounded! Thanks to the classical Rabinowitz alternative (because the unboundedness of continua has been eliminated) these continua meet the set of trivial solutions in at least two points. A general class of quasi-linear elliptic systems has been considered in \cite{HeaKie1}. It has been proved that, under  additional assumptions on nodal sets of the eigenvalues of the linearized problem, some of  the continua of nontrivial solutions are separated and that is why unbounded.

The Neumann problem on a two-dimensional ball has been considered in \cite{Miyamoto0}. It has been proved that there are unbounded continua consisting of non-radially symmetric solutions emanating from the second and third eigenvalues of the Laplace operator. The Neumann problem on a ball of any dimensions  has been studied in \cite{Miyamoto}. The author has proved the existence of an unbounded continuum of non-radially symmetric solutions of this problem bifurcating from the second eigenvalue of the Laplace operator. Under additional assumptions this continuum is unbounded in $\lambda$-direction.

A nonlinear eigenvalue problem on the sphere $S^{n-1}$ has been considered in \cite{RybickiLB}. It has been proved that any continuum of nontrivial solutions bifurcating from the trivial ones is unbounded. Similar results for the non-cooperative systems of elliptic differential equations have been obtained in  \cite{LB}.

In this article we consider weak solutions of problem \eqref{Dirichletwstep}
as orbits of critical points of an $\SO(n)$-invariant functional defined on a suitably chosen infinite-dimensional orthogonal representation of  $\SO(n).$ This justifies an application of a  special degree, i.e. the degree  for equivariant gradient maps, see \cite{degree,Rybicki}. It is worth to point out that this degree is an element of the Euler ring $U(\SO(n))$ of $\SO(n),$ see \cite{Dieck,Dieck1} for the definition of this ring. The advantage of using this degree lies in the fact that it allows us to distinguish homotopy classes of equivariant gradient maps. We emphasize that for the class of equivariant gradient maps our degree is stronger than the classical Leray-Schauder degree.

We have proved that if the geodesic ball is a hemisphere, then any continuum of weak solutions of system \eqref{Dirichletwstep}, which bifurcates from the set of trivial ones, is unbounded, see Theorem \ref{pmpn}, \ref{pmpn1}. In other words one of the Rabinowitz alternatives is eliminated, by showing that all global solution branches are unbounded.

How did we prove it? Applying the degree for strongly indefinite $\SO(n)$-invariant functional, we associate a bifurcation index, defined by formula \eqref{bifind},  to each point at which the necessary condition for a bifurcation is satisfied. Next we show that for any choice of a finite number of these indexes, their sum is nontrivial in the Euler ring $U(\SO(n))$. Hence the Symmetric Rabinowitz alternative, see Theorem \ref{alternative},  implies unboundedness of the bifurcating continua. We have received it through a careful analysis of the eigenspaces of the Laplace-Beltrami operator $\Delta_{S^n}$ as representations of  $\SO(n),$ see Remark \ref{stru}, Theorem \ref{PrzWlasneNaPolkoliD} and Corollary \ref{55}. It would be desirable to prove unboundedness of continua for a geodesic ball of any radius but we have not been able to do this.

For  geodesic ball of an arbitrary radius we have characterized bifurcation points at which  the global symmetry-breaking phenomenon occurs, see Theorem \ref{sbm} and Corollary \ref{SymBreakD}. Finally, a necessary condition for the existence of bounded continua is presented in Theorem \ref{bound}.

After this introduction our article is organized as follows.

In Section \ref{prelim}  we recall basic properties of non-cooperative elliptic systems considered on geodesic balls and eigenvalues of the Laplace-Beltrami operator on these balls. The main problem is given by formula \eqref{Dirichlet}. The associated functional is defined by formula \eqref{PhiD}.  Its properties are described in Lemma  \ref{operatorLBzD}. We introduce a notion of a local bifurcation of solutions of nonlinear problems in Definition \ref{defbif}. The set of parameters at which the bifurcation of solutions of problem \eqref{Dirichlet} can occur is described in Lemma \ref{313D} and Theorem \ref{P(Fi)}. Finally, properties of the eigenvalues and eigenspaces of the Laplace-Beltrami operator considered on geodesic balls in $S^n$ (with Dirichlet boundary conditions) are described in Remark \ref{stru}, Theorem \ref{PrzWlasneNaPolkoliD} and Corollary \ref{55}.

In Section \ref{main}   the main results of this article are stated and proved. The unboundedness of continua  of solutions of system \eqref{Dirichletwstep} on a hemisphere i.e. for $\gamma=\pi \slash 2$ are proved in Theorems \ref{pmpn}, \ref{pmpn1}. A  characterization of bifurcation points of this system at which the global symmetry-breaking of solutions phenomenon occurs on  geodesic ball of an arbitrary radius is given by  Theorem \ref{sbm} and Corollary \ref{SymBreakD}. A necessary condition for the existence of bounded continua of solutions of problem \eqref{Dirichletwstep} are proved in Theorem \ref{bound}.

In Section \ref{appendix}, for the convenience of the reader, we have repeated the relevant material on equivariant bifurcation theory thus making our exposition self-contained. For the definition of the Euler ring $U(G)$ of a compact Lie group $G$ we refer the reader to \cite{Dieck,Dieck1}. Since most of the computations in this article will be done in the Euler ring $U(\SO(2)),$ we have reminded the definition of this ring, see Definition \ref{uso2}. Next we have reminded classification of orthogonal representations of  $\SO(2),$ see  \eqref{reprso2}. A definition of the bifurcation index which is an element of the Euler ring $U(\SO(n))$ is given by formula \eqref{bifind}. Remark  \ref{iso2}
allows us to reduce difficult computations in the Euler ring $U(\SO(n))$ to much simpler computations in the Euler ring of $\SO(2).$ The essential role in the proofs of the main results of this article plays the Symmetric Rabinowitz alternative, see Theorem \ref{alternative}. In Lemmas \ref{pm}, \ref{pp} we prove formulas for bifurcation indexes.

\section{Preliminaries}
\label{prelim}

In this section we remind some definitions of equivariant topology. Moreover, we  present  a variational setting of our problem and study properties of the associated functional.

Throughout this section $G$ stands for a real compact Lie group.
Let $X$ be a topological space. An action of $G$ on $X$ is a continuous map $\rho : G \times X \to X$ such that
\begin{itemize}
\item $\rho(g,\rho(h,x))=\rho(gh,x)$ for $g,h \in G, x \in X,$
\item  $\rho(e,x)=x $ for $x \in X, e \in G$ the unit.
\end{itemize}

A $G$-space is a pair $(X,\rho)$ consisting of a space together with an action  of $G$ on $X$. Usually the $G$-space $(X,\rho)$ is denoted just by underlying topological space $X$ and $\rho(g, x)$ is denoted by $gx.$  An action of $G$ on $X$ is called trivial if $gx=x$ for all  $x \in X, g \in G.$ 
For each $x \in X$ the set $G(x)= \{gx : g \in G$ is called the orbit through
$x$ and $G_x = \{ g \in G : gx = x\}$ is called the isotropy group of
$x.$
A subset $A$ of a $G$-space $X$ is said to be  $G$-invariant if for all $x \in A$ and $g \in G$ we have $gx \in A$ i.e. $G(x) \subset A$ for any $x \in A.$

Suppose that $X$ and $Y$ are $G$-spaces. A continuous map $f : X \to Y$ is called $G$-equivariant if  for all $g \in G$ and
$x \in X$ the equality $f(gx) =  gf(x)$ holds true. Moreover, if $Y = R$ with  trivial action of $G$ then a map $f : X \to R$ is said to be
$G$-invariant.

\begin{Definition}
Let $\bV$, $\bV'$ be representations of a compact Lie group $G.$ We say that $\bV$ and $\bV'$ are equivalent if there is an equivariant linear isomorphism $L \colon  \bV \to \bV'.$ For the sake of simplicity we denote this relation briefly $\bV \approx_G \bV'.$
\end{Definition}

Throughout this article $\SO(n)$ stands for a real special orthogonal group.

Consider the sphere $S^{n} = \{x\in\bR^{n+1}\colon \|x\|=1\} $ and the metric between two points $p,q\in S^n$ defined by $d(p,q)=\inf\limits_{\omega}\int\limits_{a}^{b}|\omega'(t)|dt$, where $\omega$ ranges over all continuous, piecewise $C^1$ paths $\omega\colon[a,b]\to S^n$ for which $\omega(a)=p$, $\omega(b)=q$. Define the geodesic ball in $S^n$ centered at $N=(0,\ldots,0,1)$ and radius $\gamma\in (0,\pi)$ by
$B(\gamma)=\{q\in S^n \colon d(N,q)<\gamma\}$.
The geodesic ball $B(\gamma)$ is an $\SO(n)$-invariant subset of the representation $\bR^{n+1}$ of the  group $\SO(n)$ with the action $\SO(n)\times \bR^{n+1}\to \bR^{n+1}$ given by $(g,x)=(g,(x_1,\ldots,x_n,x_{n+1}))\mapsto (g(x_1,x_2,\ldots,x_n),x_{n+1})$.

Consider the following system of equations
\begin{equation}\label{Dirichlet}
\left\{
\begin{array}{rclll}
\Lambda \Delta_{S^{n}} u & = &\nabla_u F(u,\lambda) & \text{ in } & B(\gamma), \\ u & = & 0  & \text{ on } & \partial B(\gamma),
\end{array}
\right.
\end{equation}
where
\begin{enumerate}
\item[(A1)] $F\in C^2(\bR^p\times\bR,\bR)$, $\nabla_{u} F(u,\lambda)=(\nabla_{u_1} F(u,\lambda),\ldots,\nabla_{u_p} F(u,\lambda))$,
\item[(A2)] $\nabla_u F(u,\lambda)=\lambda u+\nabla_u \eta(u,\lambda)$, where $\eta \in C^2(\bR^p \times \bR,\bR)$ and for every $\lambda \in \bR$ follows
$\nabla_u \eta(0,\lambda)=0$ and $\nabla^2_u \eta(0,\lambda)=0$,
\item[(A3)] there exist $C>0$ and $1\leq s < (n+2)(n-2)^{-1}$ such that $|\nabla^2_u F(u,\lambda)|\leq C(1+|u|^{s-1})$ (if $n=2$, we assume that $s\in[1,+\infty))$,
\item[(A4)] $\Lambda= \diag(\alpha_1,\ldots,\alpha_p), \alpha_i\in\{-1,1\}$.
\end{enumerate}

If the coefficients $\alpha_i$ are not of the same sign,  we call  system \eqref{Dirichlet} non-cooperative.
From now on  $p_-$ ($p_+$) stands for the number of negative (positive) $ \alpha_i,$   $i=1,\ldots, p$.

Let $H^1_0(B(\gamma))$ denote the Sobolev space with the inner product
\[
\forall_{u,v\in H^1_0(B(\gamma))}\ \langle u,v\rangle_{H^1_0(B(\gamma))}=\int\limits_{B(\gamma)}\langle \nabla u(x),\nabla v(x)\rangle d\sigma.
\]
The space $H^1_0(B(\gamma))$ is an orthogonal representation of $\SO(n)$ with the action given by $\SO(n) \times H^1_0(B(\gamma)) \ni (g,u) \to gu \in H^1_0(B(\gamma)),$ where $(gu)(x)=u(g^{-1}x).$
Let $\bH$ be the direct sum of $p$ copies of the representation $H^1_0(B(\gamma))$, i.e.   $\bH=\bigoplus\limits^p_{i=1}H^1_0(B(\gamma))$. We consider $\bH\times \bR$  as a representation of $\SO(n)$ with the action given by $\SO(n) \times (\bH\times \bR) \ni (g,(u,\lambda)) \to (gu,\lambda) \in \bH\times \bR.$

Define a family of  $\SO(n)$-invariant  functionals $\Phi\colon\bH\times\bR\to\bR$ of the class $C^2$ by
\begin{equation}\label{PhiD}
\Phi(u,\lambda)=\frac{1}{2}\int\limits_{B(\gamma)}\sum^p_{i=1}(- \alpha_i|\nabla u_i(x)|^2)d\sigma-\int\limits_{B(\gamma)}F(u(x),\lambda)d\sigma
\end{equation}
and note that
\begin{eqnarray*}
 \Phi(u,\lambda)&=&\sum^p_{i=1}\frac{1}{2}\int\limits_{B(\gamma)}(- \alpha_i|\nabla u_i(x)|^2)d \sigma
 -\frac{\lambda}{2}\int\limits_{B(\gamma)} |u(x)|^2 d\sigma-
 \int\limits_{B(\gamma)}\eta(u(x),\lambda)d\sigma \\
 &=&
-\frac{1}{2}\sum^p_{i=1} \alpha_i||u_i||_{H^1_0(B(\gamma))}^2
-\frac{\lambda}{2}\int\limits_{B(\gamma)}|u_i(x)|^2d\sigma
-\int\limits_{B(\gamma)}\eta(u(x),\lambda)d\sigma.
\end{eqnarray*}
Define $T\colon H^1_0(B(\gamma))\to H^1_0(B(\gamma))$ and $\eta_0\colon \bH\times\bR\to \bR$ by
\[
\forall_{v\in H^1_0(B(\gamma))}\  \langle Tu ,v\rangle_{H^1_0(B(\gamma))}=\int\limits_{B(\gamma)}u(x)v(x)d\sigma,\ \ 
\eta_0(u,\lambda) =\int\limits_{B(\gamma)}\eta(u(x),\lambda)d\sigma.
\]

Since the functional $\Phi(\cdot,\lambda)$ is $\SO(n)$-invariant, its gradient $\nabla_u\Phi(\cdot,\lambda)$ is $\SO(n)$-equivariant.
The following lemma is a direct consequence of the above computations.

\begin{Lemma}\label{operatorLBzD}
Under the above assumptions:
\begin{eqnarray*}
\nabla_u\Phi(u,\lambda)&=&Lu-\lambda K(u)-\nabla_u\eta_0(u,\lambda)\\
&=&(- \alpha_1u_1-\lambda Tu_1,\ldots,- \alpha_p u_p-\lambda Tu_p)-\nabla_u\eta_0(u,\lambda),
\end{eqnarray*}
where
\begin{enumerate}
\item $L\colon\H\to\H$ given by $L(u_1,\ldots,u_p)=(-\alpha_1u_1,\ldots,-\alpha_pu_p)$ is a self-adjoint, bounded $\SO(n)$-equivariant Fredholm operator,
\item $K=(T,\ldots,T)\colon\bH\to\bH$ is a self-adjoint, bounded, completely continuous  $\SO(n)$-equivariant operator,
\item $\nabla_u\eta_0(u,\lambda)\colon\bH\times\bR\to\bH$ is a completely continuous, $\SO(n)$-equivariant operator such that $\nabla_u\eta_0(0,\lambda)=0$, $\nabla^2_u\eta_0(0,\lambda)=0$ for every $\lambda\in\bR$,
\item $u=(u_1,\ldots,u_p)\in \bH$ is a weak solution of system \eqref{Dirichlet} if and only if $\nabla_u\Phi(u,\lambda)=0$, that is  $u$ is a critical point of $\Phi$.
\end{enumerate}
\end{Lemma}
 Denote by $\sigma(-\Delta_{S^{n}};B(\gamma))=\{0 < \lambda_1 < \lambda_2 < \ldots\}$ the set of all eigenvalues of the problem
\begin{equation}\label{eigenDirichlet}
\left\{\begin{array}{rclcc}
-\Delta_{S^{n}} u(x)&=&\lambda u(x) &\text{ in }&B(\gamma), \\
u(x)&=&0 &\text{ on }& \partial B(\gamma),
\end{array}
\right.
\end{equation}
and by $V^{\gamma}_{-\Delta_{S^{n}}}(\lambda_{m_0})$  the eigenspace associated with the eigenvalue $\lambda_{m_0}\in\sigma(-\Delta_{S^{n}};B(\gamma)).$ Set $\sigma^-(-\Delta_{S^{n}};B(\gamma))=\{-\lambda_m \colon  \lambda_m \in \sigma(-\Delta_{S^{n}};B(\gamma))\}$ and $
\cP^{\gamma}(\Phi)=\{\lambda_0 \in \bR \colon  \nabla^2_u\Phi(0,\lambda_0)=L-\lambda_0 K \text{ is not an isomorphism}\}.
$

In the next lemma we formulate basic properties of eigenvalues and eigenspaces of the Laplace-Beltrami operator on the geodesic ball $B(\gamma).$ We omit an easy proof of this lemma.

\begin{Lemma}\label{313D}
Under the above assumptions:
\begin{enumerate}
\item
$\cP^{\gamma}(\Phi)=
\left\{
\begin{array}{ll}
\sigma(-\Delta_{S^{n}};B(\gamma)),& \text{ when } p_- = p,  \\
\sigma^-(-\Delta_{S^{n}};B(\gamma)),& \text{ when }  p_+=p, \\
\sigma(-\Delta_{S^{n}};B(\gamma))\cup \sigma^-(-\Delta_{S^{n}};B(\gamma)),& \text{ when } p_- p_+>0, \\
\end{array}
\right.$
\item $\sigma(K)=\{\frac{1}{\lambda_m}\colon\lambda_m\in\sigma(-\Delta_{S^{n}};B(\gamma))\}$,
\item $V_{K}(\frac{1}{\lambda_m})=\bigoplus\limits^{p}_{i=1}  V^{\gamma}_{-\Delta_{S^{n}}}(\lambda_{m})$,
\item for every $\lambda_m\in\sigma(-\Delta_{S^{n}};B(\gamma))$ the subspace $ V^{\gamma}_{-\Delta_{S^{n}}}(\lambda_{m}) \subset  H^1_0(B(\gamma))$ is finite dimensional,
\item  $H^1_0(B(\gamma))=\mathrm{cl}(\bigoplus\limits^{\infty}_{m=1}  V^{\gamma}_{-\Delta_{S^{n}}}(\lambda_{m}))$.
\end{enumerate}
Moreover, for $\lambda_0 \in \cP^{\gamma} (\Phi)$,
\item if $\lambda_0>0$, then $p_->0$, $\lambda_0 \in \sigma(-\Delta_{S^{n}};B(\gamma))$ and $\ker\nabla^2_u\Phi(0,\lambda_0)=\bigoplus\limits^{p_-}_{i=1} V^{\gamma}_{-\Delta_{S^{n}}}(\lambda_0)$,
\item if $\lambda_0 < 0$, then $p_+>0$, $-\lambda_0 \in \sigma(-\Delta_{S^{n}};B(\gamma))$ and $\ker\nabla^2_u\Phi(0,-\lambda_0)=\bigoplus\limits^{p_+}_{i=1} V^{\gamma}_{-\Delta_{S^{n}}}(-\lambda_0)$.
\end{Lemma}

Let us remind that $\nabla_u\Phi(0,\lambda)=0$ for every $\lambda\in\bR.$

\begin{Definition} \label{defbif} A solution $(0,\lambda)$  of the equation $\nabla_u\Phi(u,\lambda)=0$ is said to be trivial.
A point $(0,\lambda_0)$ is said to be a bifurcation point of solutions of the equation  $\nabla_u\Phi(u,\lambda)=0$ if $(0,\lambda_0)\in \mathrm{cl} \{(u,\lambda)\in \bH\times\bR\colon \nabla_u\Phi(u,\lambda)=0, u\neq 0 \}$.
\end{Definition}

In the following theorem we formulate a necessary condition for the existence of bifurcation points of solutions of equation $\nabla_u\Phi(u,\lambda)=0.$ This
theorem is a direct consequence of Lemmas \ref{operatorLBzD}, \ref{313D}  and the implicit function theorem.

\begin{Theorem}\label{P(Fi)}
If $(0,\lambda_0)$ is a bifurcation point of solutions of the equation $\nabla_u\Phi(u,\lambda)=0$, then $\lambda_0 \in \mathcal{P}^{\gamma}(\Phi)$.
\end{Theorem}

Let $(t,\theta)$ be the geodesic spherical coordinate on $S^n$.
The eigenvectors of problem \eqref{eigenDirichlet} are of the form
$
u(t,\theta)= T_m(\lambda,t) v_m(\theta),
$
see \cite{Bang, Bang1}, where $m \geq 0$ and $v_m(\theta)$ is a spherical harmonic of $n$ variables and  degree $m$, i.e. $v_m$ is a solution of the equation
\[
-\LB v(\theta)=\beta_m v(\theta), \text{ where } \beta_m=m(m+n-2),
\]
and $T_m(\lambda,t)$ is a solution of the equation
$$
T''(t)+(n-1)(\cot t) T'(t)+\left(\lambda-\frac{\beta_m}{\sin^2t}\right)T(t)=0.
$$
The explicit formula for $T_m(\lambda,t)$ is given in \cite{Bang}.

For $m \geq 0$ define $A^{\gamma}_m=\{\lambda > 0 \colon  T_m(\lambda,\gamma)=0\}.$  From the general theory of the eigenvalue problem the set  $A^{\gamma}_m$ is countable and $\displaystyle \sigma(-\Delta_{S^n};B(\gamma))= \bigcup_{m=0}^{\infty} A^{\gamma}_m,$ see  \cite{Bang}. Moreover, if $\lambda_0\in A^{\gamma}_m$, then $V^{\gamma}_{-\Delta_{S^n}}(\lambda_0)$ is equivalent as  representation of $\SO(n)$ to $\cH^n_m$ , where
$\cH^n_m$ denotes the linear space of harmonic, homogeneous polynomials of $n$ independent variables of degree $m$, restricted to the sphere $S^{n-1}$, see Section \ref{appendix}.

\begin{Remark} \label{stru}
Fix $\lambda_0 \in \sigma(-\Delta_{S^n};B(\gamma))$ and define $\Gamma^{\gamma}(\lambda_0)=\{m \geq 0 \colon  \lambda_0 \in A_m^{\gamma}\}.$ Since the multiplicity of $\lambda_0$ is finite, $\mathrm{card}(\Gamma^{\gamma} (\lambda_0)) < \infty.$ Without loss of generality one can assume that $\Gamma^{\gamma}(\lambda_0)=\{m_1, \ldots,m_q\}$ and that $0 \leq m_1 < \ldots < m_q.$ Finally we obtain that $V^{\gamma}_{-\Delta_{S^n}}(\lambda_0) \approx_{\SO(n)}  \cH^n_{m_1} \oplus \ldots \oplus \cH^n_{m_q}$, i.e.  representations $V^{\gamma}_{-\Delta_{S^n}}(\lambda_0)$   and $\cH^n_{m_1} \oplus \ldots \oplus \cH^n_{m_q}$  are $\SO(n)$-equivalent.
\end{Remark}

In the theorem below we  discuss the special case of hemisphere i.e. $\gamma=\pi/2.$

\begin{Theorem}[Theorem 6 of \cite{Bang} ]\label{PrzWlasneNaPolkoliD}
Suppose that $\gamma=\pi/2.$ Then
\begin{eqnarray*}
&&\sigma\left(-\Delta_{S^{n}};B(\pi/2)\right) =
\left\{\lambda_m=m(n+m-1)\colon m\in\bN\right\}
\end{eqnarray*}
and
$ V^{\pi /2}_{-\Delta_{S^{n}}}(\lambda_m) \approx_{\SO(n)}   \bigoplus\limits_{l\colon\exists_{p \in \bN \cup \{0\}}m=2p+l+1} \cH^n_l$.
Moreover, the multiplicity of $\lambda_m$ is $\binom{n+m-2}{n-1}$  for every $m\in\bN$.
\end{Theorem}
The following corollary is a direct consequence of the above theorem.

\begin{Corollary}\label{55}
Fix $\lambda_m\in \sigma(-\Delta_{S^{n}};B(\pi/2))$. Then
\begin{enumerate}[1)]
\item if $m$ is even, then
$ V^{\pi/2}_{-\Delta_{S^{n}}}(\lambda_m) \approx_{\SO(n)}    \cH^n_{1}\oplus\cH^n_{3}\oplus\ldots\oplus\cH^n_{m-1}$,
\item if $m$ is odd, then
$ V^{\pi/2}_{-\Delta_{S^{n}}}(\lambda_m) \approx_{\SO(n)}   \cH^n_{0}\oplus\cH^n_{2}\oplus\ldots\oplus\cH^n_{m-1}$.
\end{enumerate}
Consequently,
$\cH^n_{m-1}\subset V^{\pi/2}_{-\Delta_{S^{n}}}(\lambda_m)$ and $\cH^n_{m-1}\not\subset V^{\pi/2}_{-\Delta_{S^{n}}}(\lambda_{\widehat{m}})$ for every $0<\widehat{m}<m$.
\end{Corollary}

\section{Main results}
\label{main}

In this section we study continua of weak solutions of non-cooperative elliptic systems considered on a geodesic ball $B(\gamma)$, where $\gamma\in(0,\pi)$.
 Consider the system

\begin{equation}\label{Dirichlet1}
\left\{
\begin{array}{rclll}
\Lambda \Delta_{S^{n}} u & = &\nabla_u F(u,\lambda) & \text{ in } & B(\gamma), \\ u & = & 0  & \text{ on } & \partial B(\gamma),
\end{array}
\right.
\end{equation}
where $F$ and  $ \Lambda$ satisfy assumptions (A1)-(A4) of Section \ref{prelim}.
 Recall that $u\in\bH=\bigoplus\limits_{i=1}^p H^1_0(B(\gamma))$ is a weak solution of the above system if and only if $u$ is a critical point of the  functional $\Phi \colon  \bH\times\bR\to\bR$  given by formula \eqref{PhiD}. That is why we study in this section solutions of the equation $\nabla_u \Phi(u,\lambda)=0.$
Denote by $C(\lambda_0)\subset\bH\times\bR$  the continuum of
$
\mathrm{cl} \{(u,\lambda)\in \bH\times\bR\colon \nabla_u\Phi(u,\lambda)=0, u \neq 0 \}
$
containing $(0,\lambda_0)$.

We  first prove unboundedness of continua of weak solutions of system \eqref{Dirichlet1} for $\gamma=\frac{\pi}{2},$ bifurcating from the set of trivial ones.
In other words  we  show that the second possibility in the Symmetric Rabinowitz alternative, see   Theorem \ref{alternative}, is eliminated. More precisely, we will prove that  formula \eqref{suml} is never satisfied.

 Let $\mu_m=\dim V^{\pi/2}_{-\Delta_{S^{n}}}(\lambda_{m})$ and $\nu_m=\mu_1+\ldots+\mu_{m}$  for $m\in\bN$.

\begin{Theorem}\label{pmpn}
Fix $\lambda_{m_0} \in \sigma(-\Delta_{S^n};B(\pi/2)) \setminus \{\lambda_1\}$. Then the continuum $C(\pm \lambda_{m_0}) \subset  \bH\times\bR$ of weak solutions of system \eqref{Dirichlet1} is unbounded.
\end{Theorem}
\begin{proof} We prove the unboundedness of the continuum $C(\lambda_{m_0}).$ The proof for the continuum $C(-\lambda_{m_0})$ is similar and left to the reader.
Since $\lambda_{m_0} \neq \lambda_1,$ from Theorem \ref{nieprzywiedlnosc} it follows that $V_{-\Delta_{S^n}}(\lambda_{m_0})$ is a nontrivial representation of  $\SO(n).$
Suppose, contrary to our claim, that the continuum  $C(\lambda_{m_0}) \subset \bH \times\bR$ is bounded. Then from the Symmetric Rabinowitz alternative, see Theorem \ref{alternative}, it follows that
\begin{enumerate}
\item $C(\lambda_{m_0}) \cap (\{0\} \times \bR)=\{0\} \times \{\widehat{\lambda}_1,  \ldots, \widehat{\lambda}_s\} \subset \{0\} \times \cP^{\pi/2}(\Phi),$
\item $\displaystyle \sum_{j=1}^s \BIF_{\SO(n)}(\widehat{\lambda}_j)=\Theta\in U(\SO(n))$.
\end{enumerate}
Without loss of generality one can assume that $\widehat{\lambda}_1 < \ldots < \widehat{\lambda}_{s'} < 0 <  \widehat{\lambda}_{s'+1} < \ldots < \widehat{\lambda}_s.$  Since $\{\widehat{\lambda}_1,  \ldots, \widehat{\lambda}_s\}  \subset \cP(\Phi),$ there are $\lambda_{m_1}, \ldots, \lambda_{m_{s'}}, \lambda_{m_{s'+1}}, \ldots,\lambda_{m_s} \in \sigma(-\Delta_{S^n};B(\frac{\pi}{2}))$ such that $\widehat{\lambda}_j=-\lambda_{m_j}$ for $j=1, \ldots,s'$ and $\widehat{\lambda}_j=\lambda_{m_j}$ for $j=s'+1,\ldots,s$, i.e.
$$-\lambda_{m_1} < \ldots < -\lambda_{m_{s'}} < 0 < \lambda_{m_{s'+1}} < \ldots < \lambda_{m_s}$$ and
$m_1> \ldots > m_{s'}$, $m_s > \ldots >m_{s'+1}.$

\noindent By Remark \ref{iso2} we obtain
$\displaystyle \sum_{j=1}^s \BIF_{\SO(2)}(\widehat{\lambda}_j) = i^{\star}\left(\sum_{j=1}^s \BIF_{\SO(n)}(\widehat{\lambda}_j)  \right)
=\Theta \in U(\SO(2)).
$
That is why  we obtain the following equality
\begin{equation} \label{sum} \sum_{j=1}^{s'} \BIF_{\SO(2)}(-\lambda_{m_j}) + \sum_{j=s'+1}^{s} \BIF_{\SO(2)}(\lambda_{m_j}) = \Theta \in U(\SO(2)).
\end{equation}
What is left is to show that the above equality is never satisfied.
In the rest of the proof  we consider four cases.

\noindent {\bf Case:} $p_-, p_+ \in 2\bN.$  Since $p_-, p_+$ are even, from Lemmas  \ref{pm}\eqref{pm2}, \ref{pp}\eqref{pp2} it follows that
$$\BIF_{\SO(2)}(\lambda_{m_{s'+1}}), \ldots, \BIF_{\SO(2)}(\lambda_{m_{s}}) \in U_-(\SO(2))
$$
and that
$$ \BIF_{\SO(2)}(-\lambda_{m_1}),\ldots, \BIF_{\SO(2)}(-\lambda_{m_{s'}})  \in U_-(\SO(2)).$$

\noindent To complete the proof it is enough to note that $ \BIF_{\SO(2)}(\lambda_{m_s})   \in U_-(\SO(2)) \setminus \{\Theta\}.$ Indeed from   Lemma \ref{pm}\eqref{pm1} it follows that $\BIF_{\SO(2)}(\lambda_{m_s}) \neq \Theta,$ which contradicts equality \eqref{sum}.

\noindent {\bf Case:} $p_- \in 2\bN, p_+ \in 2\bN+1.$  Since $p_-$ is even, from Lemmas \ref{pm}\eqref{pm1}, \ref{pm}\eqref{pm2} we obtain  that   $\BIF_{\SO(2)}(\lambda_{m_{s'+1}}), \ldots, \BIF_{\SO(2)}(\lambda_{m_{s}}) \in U_-(\SO(2)) \setminus \{\Theta\}.$ From Lemma \ref{pm}\eqref{pm1}  we obtain $\BIF_{\SO(2)}(\lambda_{m_s}) =(1,\alpha_1,\ldots, \alpha_{m_s-2},-p_-, 0,\ldots) \in U_-(\SO(2)).$ Moreover, for $j=s'+1,\ldots,s-1$ if $\BIF_{\SO(2)}(\lambda_{m_j}) =(\alpha_0,\alpha_1, \ldots, \alpha_k,\ldots)$ then $\alpha_k = 0, \text{ for } k \geq m_s-1.$  From Lemma \ref{pp}\eqref{pp1} we obtain $\BIF_{\SO(2)}(-\lambda_{m_1}) =(\alpha_0  ,\alpha_1,\ldots, \alpha_{m_1-2},(-1)^{1+\nu_{m_1}p_+} p_+, 0,\ldots) \in U(\SO(2)).$  Moreover, for $j=2,\ldots,s'$ if $\BIF_{\SO(2)}(-\lambda_{m_j}) =(\alpha_0,\alpha_1,\ldots, \alpha_k,\ldots)$ then $\alpha_k = 0, \text{ for } k \geq m_1-1.$  From the above reasoning and formula \eqref{sum} it follows that $m_1=m_s$ and that the $m_s$-th coord
 inate of formula \eqref{sum} equals $-p_- +  (-1)^{1+\nu_{m_1}p_+} p_+=0.$ Since $p_- \neq p_+$, formula \eqref{sum} is not fulfilled, a contradiction.

\noindent {\bf Case:} $p_- \in 2\bN+1, p_+ \in 2\bN.$  A proof is in fact the same as the proof of the previous case.

\noindent {\bf Case:} $p_-, p_+ \in 2\bN+1.$  In the first case we have considered the numbers $p_-$, $ p_+$ of the same even parity.  Now  the numbers $p_-$, $p_+$ are of the same but odd parity. In this case the bifurcation indexes are not elements of $U_-(\SO(2)).$ Taking into account Lemmas \ref{pm}\eqref{pm1}, \ref{pp}\eqref{pp1} and formula \eqref{sum} we obtain $m_1=m_s.$ Moreover, the  $m_s$-th coordinate of formula \eqref{sum}  has the following form
\begin{equation} \label{mssum}
(-1)^{1+\nu_{m_1}p_-} p_- + (-1)^{1+\nu_{m_1}p_+} p_+=0.
\end{equation}
 Thus we obtain $p_-=-p_+,$   a contradiction.
\end{proof}

In the  theorem below we describe continua $C(\pm \lambda_1) \subset \bH \times \bR$ of weak  solutions of system \eqref{Dirichlet1}, i.e. continua bifurcating from the first eigenvalue $\pm \lambda_1.$

\begin{Theorem} \label{pmpn1}
 If  $p_{\mp}$  is odd and $\lambda_1 \in \sigma(-\Delta_{S^n};B(\pi/2)) $, then the continuum  $C(\pm\lambda_1) \subset \bH \times \bR$ of weak solutions of system \eqref{Dirichlet1} is unbounded.
\end{Theorem}
\begin{proof} We prove this theorem for $p_- > 0.$ The proof for $p_+ > 0$ is similar and left to the reader.
From Lemma \ref{pm}\eqref{pm1} it follows that
$\BIF_{\SO(2)}(\lambda_1)=(-2,0,\ldots,0,\ldots) \in U(\SO(2)).$
Suppose contrary to our claim that the continuum $C(\lambda_1)$ is bounded.
Then by the Symmetric Rabinowitz alternative, see Theorem \ref{alternative}, the continuum $C(\lambda_1)$ meets the set of trivial solutions $\{0\} \times \bR \subset \bH \times \bR$ at a finite number of points. By Theorem  \ref{pmpn} the continua $C(\pm \lambda_m)$, $m >1,$
are unbounded. Therefore  $C(\lambda_1)\cap(\{0\}\times\bR) =\{0\}\times\{-\lambda_{1},\lambda_{1}\}$. Moreover,
 $\BIF_{\SO(n)}(\lambda_{1})+\BIF_{\SO(n)}(-\lambda_{1})=\Theta\in U(\SO(n))$ and consequently
\begin{equation} \label{sum1}\BIF_{\SO(2)}(\lambda_{1})+\BIF_{\SO(2)}(-\lambda_{1}) = i^{\star}\left(\BIF_{\SO(n)}(\lambda_{1})+\BIF_{\SO(n)}(-\lambda_{1})\right)= \Theta\in U(\SO(2)).
\end{equation}
By Lemma \ref{pp}\eqref{pp1} we have $\BIF_{\SO(2)}(-\lambda_1)=((-1)^{p_+}-1,0,\ldots,0,\ldots) \in U(\SO(2)).$
Therefore
\[
\BIF_{\SO(2)}(\lambda_1)_{\SO(2)}+\BIF_{\SO(2)}(-\lambda_1)_{\SO(2)}=
-2+(-1)^{p_+}-1=-3+(-1)^{p_+}\neq 0,
\]
which contradicts equality \eqref{sum1}.
\end{proof}

From now on we consider system \eqref{Dirichlet1} on a geodesic ball $B(\gamma)\subset S^n$ with $\gamma\in(0,\pi)$. Since in this case the structure of the eigenspaces as representations of  $\SO(2)$ are not known explicitly, the reasoning from the proofs of Theorems \ref{pmpn}, \ref{pmpn1}  cannot be repeated.

In the theorem below we formulate necessary  conditions for boundedness of continua of weak solutions of system \eqref{Dirichlet1}.

\begin{Theorem} \label{bound}
Let $\lambda_{m_0} \in \sigma(-\Delta_{S^{n}};B(\gamma)) \setminus \{\lambda_1\}.$
Then if $p_{\mp}>0$ is even and the continuum $C(\pm \lambda_{m_0}) \subset \bH \times\bR$  is bounded, then $p_{\pm}>0$ is odd and
$
C(\pm \lambda_{m_0}) \cap (\{0\} \times \sigma^{\mp}(-\Delta_{S^{n}};B(\gamma))) \neq \emptyset.
$
\end{Theorem}
\begin{proof}
We prove this theorem for even $p_- >0.$ The proof for even $p_+ >0$ is similar and left to the reader. Suppose, contrary to our claim, that $p_- > 0$ is even, the continuum $C( \lambda_{m_0}) \subset \bH \times\bR$  is bounded, $p_+ > 0$ is even or $C(\lambda_{m_0}) \cap (\{0\} \times \sigma^-(-\Delta_{S^{n}};B(\gamma))) = \emptyset.
$

Since $\lambda_{m_0} \neq \lambda_1,$ combining Lemma \ref{313D}, Remark \ref{stru} with Theorems \ref{shima}, \ref{nieprzywiedlnosc} we obtain that $V^{\gamma}_{-\Delta_{S^n}}(\lambda_{m_0})$ is a nontrivial representation of  $\SO(n).$
Since the continuum  $C(\lambda_{m_0}) \subset \bH \times\bR$ is bounded, from  the Symmetric Rabinowitz alternative, see Theorem \ref{alternative}, it follows that
\begin{enumerate}
\item $C(\lambda_{m_0}) \cap (\{0\} \times \bR)=\{0\} \times \{\widehat{\lambda}_1,  \ldots, \widehat{\lambda}_s\} \subset \{0\} \times \cP^{\gamma}(\Phi),$
\item $\displaystyle \sum_{j=1}^s \BIF_{\SO(n)}(\widehat{\lambda}_j)=\Theta\in U(\SO(n))$.
\end{enumerate}
Without loss of generality one can assume that $\widehat{\lambda}_1 < \ldots < \widehat{\lambda}_{s'} < 0 <  \widehat{\lambda}_{s'+1} < \ldots < \widehat{\lambda}_s.$  Since $\{\widehat{\lambda}_1,  \ldots, \widehat{\lambda}_s\}  \subset \cP^{\gamma}(\Phi),$ there are $\lambda_{m_1}, \ldots, \lambda_{m_{s'}}, \lambda_{m_{s'+1}}, \ldots,\lambda_{m_s} \in \sigma(-\Delta_{S^n};B(\gamma))$ such that $\widehat{\lambda}_j=-\lambda_{m_j}$ for $j=1, \ldots,s'$ and $\widehat{\lambda}_j=\lambda_{m_j}$ for $j=s'+1,\ldots,s$, i.e. $$-\lambda_{m_1} < \ldots < -\lambda_{m_{s'}} < 0 < \lambda_{m_{s'+1}} < \ldots < \lambda_{m_s}$$ and $m_1> \ldots > m_{s'}, m_s > \ldots >m_{s'+1}.$

\noindent By Remark \ref{iso2} we obtain
$\displaystyle \sum_{j=1}^s \BIF_{\SO(2)}(\widehat{\lambda}_j) = i^{\star}\left(\sum_{j=1}^s \BIF_{\SO(n)}(\widehat{\lambda}_j)  \right)
=\Theta \in U(\SO(2)).
$
That is why  we obtain the following equality
\begin{equation} \label{suma} \sum_{j=1}^{s'} \BIF_{\SO(2)}(-\lambda_{m_j}) + \sum_{j=s'+1}^{s} \BIF_{\SO(2)}(\lambda_{m_j}) = \Theta \in U(\SO(2)).
\end{equation}
Since $p_-$ is even, taking into account Lemmas \ref{pm}\eqref{pm1},  \ref{pm}\eqref{pm2} we obtain
\begin{equation} \label{suma1} \sum_{j=s'+1}^{s} \BIF_{\SO(2)}(\lambda_{m_j})  \in U_-(\SO(2)) \setminus \{\Theta\}.
\end{equation}
Comparing formulas  \eqref{suma} and \eqref{suma1} we obtain that $
C( \lambda_{m_0}) \cap (\{0\} \times \sigma^-(-\Delta_{S^{n}};B(\gamma))\}) \neq \emptyset$ and  $p_+ >0$ is odd. Indeed if $p_+$ is even then by Lemma \ref{pp}\eqref{pp2} we obtain
\begin{equation} \label{suma2} \sum_{j=1}^{s'} \BIF_{\SO(2)}(-\lambda_{m_j}) \in U_-(\SO(2)).
\end{equation}
But formulas \eqref{suma1}, \eqref{suma2} contradict formula \eqref{suma}, which implies that $p_+$ is odd, a contradiction.
\end{proof}

\begin{Definition}
We say that $(0,\lambda_{m_0})\in\bH\times\bR$  is a global symmetry-breaking bifurcation point of solutions of system \eqref{Dirichlet1} if there exists an open $\SO(n)$-invariant neighborhood $U\subset\bH\times\bR$  of $(0,\lambda_{m_0})$ such that the isotropy group $\SO(n)_{(u,\lambda)}$ of every element $(u,\lambda)\in (U\cap C(\lambda_{m_0}))\setminus(\{0\}\times\bR)$ is different from $\SO(n)$.
\end{Definition}

In the theorem below we characterize global symmetry-breaking points of weak solutions of system \eqref{Dirichlet1}.

\begin{Theorem}  \label{sbm}
Assume that  $\lambda_{m_0} \in \sigma(-\Delta_{S^n},B(\gamma)) \setminus A^{\gamma}_0.$ If $p_{\mp}>0$, then $(0,\pm \lambda_{m_0}) \in \bH \times \bR$  is a global symmetry-breaking bifurcation point of solutions of system \eqref{Dirichlet1}.
\end{Theorem}
\begin{proof} We prove this theorem for $p_- > 0.$ The proof for $p_+ >0$ is similar and left to the reader.
Since $\lambda_{m_0} \notin A^{\gamma}_0,$ from  Remark \ref{stru} we obtain $\Gamma^{\gamma}(\lambda_{m_0}) = \{m_1,\ldots,m_q\}$ and $0 < m_1 < \ldots < m_q.$ Moreover,
$V^{\gamma}_{-\Delta_{S^n}}(\lambda_{m_0}) \approx_{\SO(n)}  \cH^n_{m_1} \oplus \ldots \oplus \cH^n_{m_q}.$ From Lemma \ref{313D} we obtain  $\ker\nabla^2\Phi(0,\lambda_{{m_0}})= \bigoplus\limits^{p_-}_{i=1}V^{\gamma}_{-\Delta_{S^n}}(\lambda_{{m_0}}).$ Summing up,  we obtain $\ker\nabla^2\Phi(0,\lambda_{{m_0}})= \bigoplus\limits^{p_-}_{i=1} \left( \cH^n_{m_1} \oplus \ldots \oplus \cH^n_{m_q} \right)$.
 Since $(\cH^n_m)^{\SO(n)} = \{0\}$ for every $m > 0,$ $$\left(\ker\nabla^2\Phi(0,\lambda_{{m_0}})\right)^{\SO(n)}= \bigoplus\limits^{p_-}_{i=1} \left( \left(\cH^n_{m_1}\right)^{\SO(n)} \oplus \ldots \oplus \left(\cH^n_{m_q}\right)^{\SO(n)} \right) =\{0\}.$$
The rest of the proof is a consequence of Theorem \ref{sb}.
\end{proof}

From Remark \ref{stru} and Corollary \ref{55} follows that if $\lambda_{m_0} \in  \sigma(-\Delta_{S^n},B(\frac{\pi}{2})) $ is such that $m_0$ is even, then $\lambda_{m_0}\notin A_0^{\pi/2}$. Therefore in the case $\gamma=\frac{\pi}{2}$, from the theorem above we obtain the following corollary:

\begin{Corollary}\label{SymBreakD}
Fix $\lambda_{m_0} \in \cP^{\pi/2}(\Phi)$. If $m_0 \in \bZ$ is even, then the point $(0, \lambda_{m_0})\in\bH\times\bR$ is a global symmetry-breaking point of weak solutions of system \eqref{Dirichlet1}.
\end{Corollary}

\section{Appendix}
 \label{appendix}

In this section, to make this article self-contained, we present all the material concerning  equivariant bifurcation theory which we need in the proofs of results of this paper.

\begin{Definition}
\label{uso2}
The Euler ring of  $\SO(2)$ is defined by $U(\SO(2))=\bZ \oplus \bigoplus\limits^{\infty}_{i=1} \bZ$ and for
$
 a=(a_0,a_1, a_2, \ldots), b=(b_0, b_1,b_2, \ldots) \in
\bZ \oplus \bigoplus\limits^{\infty}_{i=1}\bZ
$
we put
\begin{eqnarray}
a+b & = &(a_0+b_0, a_1+b_1, a_2+b_2, \ldots), \\
\label{mult}
a * b & = & (a_0b_0,a_1b_0+a_0b_1, a_2b_0+a_0b_2, \ldots).
\end{eqnarray}
The element $\Theta=(0,0,0,\ldots)\in U(\SO(2))$ is the neutral element and $\I=(1,0,0,\ldots)\in U(\SO(2))$ is the unit.
\end{Definition}
The definition above agrees with that of \cite{Dieck, Dieck1}, where one can find further information,
in particular the definition of the Euler  ring $U(\SO(n))$ for $n\geq 2$.

For $a=(a_0,a_1, a_2, \ldots) \in U(\SO(2))$, $a_0$ corresponds to the isotropy group $\SO(2)$ and $a_i$ to the group $\bZ_i\subset\SO(2)$, which is isomorphic to the cyclic subgroup of $S^1$, for $i\in\bN$.

Put
\begin{equation*}
\begin{array}{ll}
U_+(\SO(2))=\{(a_0,a_1, a_2, \ldots) \in U(\SO(2))\colon \forall_{i\in\bN\cup\{0\}}\ a_i \geq 0\},\\
U_-(\SO(2))=\{((a_0,a_1, a_2, \ldots) \in U(\SO(2)) \colon  \forall_{i\in\bN\cup\{0\}}\ a_i \leq 0\}.
\end{array}
\end{equation*}
The degree for  $\SO(n)$-invariant strongly indefinite functionals is an element of the Euler ring $U(\SO(n)),$ see  \cite{degree} for the definition. For the general theory of the equivariant degree we refer the reader to \cite{BKS}.

Let $m\in\bN$ and denote by $\bR[1,m]$ the two-dimensional representation of  $\SO(2)$ with a linear $\SO(2)$-action defined by
$$\left(\left[\begin{array}{rr} \cos \varphi & - \sin \varphi \\ \sin \varphi & \cos \varphi\end{array} \right], \left[\begin{array}{cc} x \\ y  \end{array}\right]\right)
\mapsto
\left[\begin{array}{rr} \cos m \varphi & - \sin m \varphi \\ \sin m\varphi & \cos m\varphi\end{array} \right] \left[\begin{array}{cc} x \\ y  \end{array}\right].$$
Note that if  $v\in\bR[1,m]\setminus\{0\}$ then the  isotropy group $\SO(2)_v=\{g\in \SO(2) \colon g v=v\}$  equals $\bZ_m$. For $k,m\in\bN$ we will denote by $\bR[k,m]$ the direct sum of $k$ copies of the representation $\bR[1,m]$. For $k\in\bN$ we denote by $\bR[k,0]$ the trivial $k$-dimensional representation of $\SO(2).$

It is known that any finite-dimensional, orthogonal representation $\bV$ of  $\SO(2)$  is equivalent to the representation of the form $\bR[k_0,0]\oplus\bR[k_1,m_1]\oplus\ldots\oplus\bR[k_r,m_r]$.
Therefore without loss of generality one can assume that
\begin{equation}
\label{reprso2}
\bV=\bR[k_0,0] \oplus \bR[k_1,m_1] \oplus \ldots \oplus\bR[k_r,m_r].
\end{equation}

Below we present the formula for the degree of $\SO(2)$-equivariant gradient maps of the map $-\Id \colon  (B(\bV),S(\bV)) \to (\bV,\bV \setminus \{0\}),$ where $B(\bV)$ is an open disc in $\bV$ of radius $1$ centered at the origin.
Namely, it is known that
$\nabla_{\SO(2)}\text{-}\deg(-\Id,B(\bV))=(\alpha_0,\alpha_1,\ldots,\alpha_i, \ldots) \in U(\SO(2)),$ where
\begin{equation}\label{stopienId}
\begin{array}{ll}
\alpha_i=
\left\{
\begin{array}{ll}
(-1)^{k_0}&\text{ if } i=0,\\
(-1)^{k_0+1}k_p&\text{ if } i=m_p, p=1,\ldots,r,\\
0&\text{ for } i \notin \{0,m_1,\ldots,m_r\}.
\end{array}
\right.
\end{array}
\end{equation}

With the functional $\Phi$ given by \eqref{PhiD} we assign a bifurcation index in terms of the degree for $\SO(n)$-equivariant strongly indefinite functionals,
see \cite{degree}.
Fix $\lambda_0 \in \cP (\Phi)$ and define the $\SO(n)$-bifurcation index $\BIF_{\SO(n)}(\lambda_0) \in U(\SO(n))$ by
\begin{equation} \label{bifind}
\BIF_{\SO(n)}(\lambda_0)=\nabla_{\SO(n)}\text{-}\deg(\nabla\Phi(\cdot,\lambda_0+\epsilon), B(\bH))-\nabla_{\SO(n)}\text{-}\deg(\nabla\Phi(\cdot,\lambda_0-\epsilon), B(\bH)),
\end{equation}
where $\epsilon>0$ is sufficiently small.

\begin{Remark} \label{iso2}
The natural inclusion $i \colon  \SO(2) \to \SO(n)$ defined  by $i(g)=\left[\begin{array}{cc} g & 0 \\ 0 & \mathrm{Id_{n-2}} \end{array} \right]$ induces a ring homomorphism $i^{\star} \colon  U(\SO(n))\to U(\SO(2)).$
We define the $\SO(2)$-bifurcation index $\BIF_{\SO(2)}(\lambda_0) \in U(\SO(2))$ by
$
\BIF_{\SO(2)}(\lambda_0)=i^{\star}(\BIF_{\SO(n)}(\lambda_0)).$
It is easy to see that
$$i^{\star}(\BIF_{\SO(n)}(\lambda_0))=\nabla_{\SO(2)}\text{-}\deg(\nabla\Phi(\cdot,\lambda_0+\epsilon), B(\bH))-\nabla_{\SO(2)}\text{-}\deg(\nabla\Phi(\cdot,\lambda_0-\epsilon), B(\bH)).
$$
\end{Remark}

The following theorem is a symmetric version of the famous Rabinowitz alternative, see \cite{Rabinowitz, Rabinowitz1}, which says that a change of the Leray-Schauder degree (non-triviality of a bifurcation index) along the line of trivial solutions implies a global bifurcation of solutions of a nonlinear eigenvalue problem.
The proof of this theorem is standard, see for instance \cite{RB,KD,JI,Nierenberg,Rabinowitz,Rabinowitz1}.

Since  $\nabla_u \Phi(\cdot,\lambda)$ is a family of strongly-indefinite $\SO(n)$-equivariant operators,  it is enough to replace in the classical proof the Leray-Schauder degree by the degree for $\SO(n)$-invariant strongly indefinite functionals, see \cite{degree}.

Finally note that under  assumptions of the following theorem for  $\lambda_{m_0}\in\sigma(-\Delta_{S^{n}};B(\gamma))$  the bifurcation indexes $\BIF_{\SO(n)}(\lambda_{m_0}), \BIF_{\SO(n)}(-\lambda_{m_0}) \in U(\SO(n))$ are  nontrivial. This is a consequence of Lemmas \ref{pm}, \ref{pp}.

\begin{Theorem}[Symmetric Rabinowitz alternative] \label{alternative}
Fix $\lambda_{m_0}\in\sigma(-\Delta_{S^{n}};B(\gamma)).$ Then
\begin{enumerate}
\item[$(p_-)$]
if   $V^{\gamma}_{-\Delta_{S^n}}(\lambda_{m_0})$ is a nontrivial representation of $\SO(n)$ or $p_- \cdot \dim V^{\gamma}_{-\Delta_{S^n}}(\lambda_{m_0})$ is odd
then  either $C(\lambda_{m_0})$ is unbounded in $\bH\times\bR$ or
\begin{enumerate}[1)]
\item $C(\lambda_{m_0})\subset\bH\times\bR$ is bounded,
\item $C(\lambda_{m_0}) \cap (\{0\}\times\bR)=\{0\}\times\{\widehat{\lambda}_1,\ldots, \widehat{\lambda}_s\} \subset \{0\} \times \cP(\Phi)$, and
\begin{equation} \label{suml} \BIF_{\SO(n)}(\widehat{\lambda}_1)+\ldots+\BIF_{\SO(n)}(\widehat{\lambda}_s)=\Theta\in U(\SO(n)).
\end{equation}
\end{enumerate}
\item[$(p_+)$] if $V^{\gamma}_{-\Delta_{S^n}}(\lambda_{m_0})$ is a nontrivial representation of $\SO(n)$ or $p_+ \cdot \dim V^{\gamma}_{-\Delta_{S^n}}(\lambda_{m_0})$ is odd,
then  either $C(-\lambda_{m_0})$ is unbounded in $\bH\times\bR$ or
\begin{enumerate}[1)]
\item $C(-\lambda_{m_0})\subset\bH\times\bR$ is bounded,
\item $C(-\lambda_{m_0})\cap(\{0\}\times\bR)=\{0\}\times\{\widehat{\lambda}_1,\ldots, \widehat{\lambda}_s\} \subset \{0\} \times \cP(\Phi)$, and
\begin{equation} \label{summl} \BIF_{\SO(n)}(\widehat{\lambda}_1)+\ldots+\BIF_{\SO(n)}(\widehat{\lambda}_s)=\Theta\in U(\SO(n)).
\end{equation}

\end{enumerate}
\end{enumerate}
\end{Theorem}

To characterize bifurcation points of system \eqref{Dirichlet1} at which the symmetry-breaking phenomenon occurs we use the following theorem. Here we locally control the isotropy groups of the bifurcating solutions by the isotropy groups   of elements of kernel. The proof of this theorem is a natural application of the Lyapunov-Schmidt reduction, it can be found for instance in \cite{Dancer}.

\begin{Theorem}\label{sb}
Let $\lambda_0 \in \cP^{\gamma}(\Phi).$ Then there exists an open $\SO(n)$-invariant neighborhood   $U \subset \bH \times \bR$ of $(0,\lambda_0)$ such that  for all $(\widehat{u},\lambda)\in (U\cap (\nabla_u \Phi)^{-1}(0)) \setminus (\{0\}\times\bR)$ there exists $\overline{u}\in \ker \nabla^2_u \Phi(0,\lambda_0) \setminus \{0\}$ such that ${\SO(n)}_{\widehat{u}}={\SO(n)}_{\overline{u}}.$
Moreover, if $\ker \nabla^2_u\Phi(0,\lambda_{m_0})^{\SO(n)}=\{0\}$, then for all $(\widehat{u},\lambda)\in (U \cap (\nabla_u\Psi)^{-1}(0)) \setminus (\{0\}\times\bR)$, ${\SO(n)}_{\widehat{u}}\neq {\SO(n)}$.
\end{Theorem}

 In the theorem below we formulate the basic properties of $\Delta_{S^{n-1}}.$ Recall that $\cH^n_m$ denotes the linear space of harmonic, homogeneous polynomials of $n$ independent variables of degree $m$, restricted to the sphere $S^{n-1}$.

\begin{Theorem}[Theorem 4.1 of \cite{Shimakura}] \label{shima} The eigenvalues of $\Delta_{S^{n-1}}$ are the following $$\lambda_m=m(m+n-2), m =0,1,2, \ldots.$$
If $V_{-\Delta_{S^{n-1}}}(\lambda_m)$ is the eigenspace of  $-\Delta_{S^{n-1}}$ belonging to $\lambda_m$ then
\begin{enumerate}
\item $ V_{-\Delta_{S^{n-1}}}(\lambda_m) = \cH^n_m,$
\item $\displaystyle \dim V_{-\Delta_{S^{n-1}}}(\lambda_m)=
\left\{\begin{array}{ccc} 1& \mathrm{ if } & n=2, m=0,
 \\ 2 & \mathrm{if} & n=2, m \geq 1,
  \\ \displaystyle (2m+n-2)\frac{(n-3+m)!}{m!(n-2)!} & \mathrm{if} & n \geq 3, m \geq 0,\end{array}  \right.$
\item $\displaystyle  L^2(S^{n-1})=\mathrm{cl}\left(\bigoplus_{m=0}^{\infty}  V_{-\Delta_{S^{n-1}}}(\lambda_m)\right).$
\end{enumerate}
\end{Theorem}

The space $\cH^n_m$ one can consider as a representation of $\SO(n)$ with the action given by the formula $\SO(n) \times \cH^n_m \ni (g,u(x))\to u(g^{-1}x) \in \cH^n_m.$

\begin{Theorem}[Theorem 5.1 of  \cite{Gurarie}] \label{nieprzywiedlnosc}
For every $m \geq 1$ the space $\cH^n_m$ is a nontrivial,  irreducible representation of  $\SO(n)$. Moreover, the space $\cH^n_0$ is a trivial representation.
\end{Theorem}
A proof of the above theorem one can find also in \cite{Vilenkin}.


The space $\cH^n_m$ one can consider as a representation of $\SO(2)$ with the action given by  $\SO(2) \times \cH^n_m \ni (g(\phi),u(x))   \to u(\widehat{g}(\phi)^{-1}x)=u(\widehat{g}(-\phi)x) \in \cH^n_m.$
In other words if $u(x_1,\ldots,x_n) \in \cH^n_m$, then
\begin{equation} \label{so2a} (g(\phi), u) (x_1,\ldots, x_n)=u(x_1 \cos \phi +x_2 \cos \phi,-x_1 \sin \phi+x_2 \cos \phi, x_3, \ldots,x_n).
\end{equation}

To calculate equivariant bifurcation indexes we will use some properties of $\cH^n_m$ as representations of  $\SO(2)$. Let us remind that  spherical coordinates have the following form
$$\begin{array}{rcl} x_1 & = & \sin \theta_{n-1} \ldots \sin \theta_2 \sin \theta_1, \\
x_2 & = & \sin \theta_{n-1} \ldots \sin \theta_2 \cos \theta_1, \\
 & \vdots & \\
 x_{n-1} & = & \sin \theta_{n-1} \cos \theta_{n-2},\\
 x_n & = & \cos \theta_{n-1},
 \end{array}$$
where $0 \leq \theta_1 < 2 \pi$, $0 \leq \theta_k < \pi$, $k \neq 1.$

\begin{Lemma}[Chapter IX of \cite{Vilenkin}] An orthonormal basis of $\cH^n_m$, $n \geq 3$, $m>0$ is given by polynomials of the form
$$ C_M(\theta_2,\ldots,\theta_{n-1})   \cos (m_{n-2}\theta_1),\ C_M(\theta_2,\ldots,\theta_{n-1})   \sin (m_{n-2} \theta_1),$$
where $M=(m_0,\ldots,m_{n-3}, m_{n-2}),$ $m=m_0 \geq m_1 \geq \ldots \geq m_{n-2} \geq 0$.
\end{Lemma}

\begin{Corollary}
\label{vilen}
Note that since $$(g(\phi), C_M(\theta_2,\ldots,\theta_{n-1})   \cos (m_{n-2} \theta_1))= C_M(\theta_2,\ldots,\theta_{n-1})   \cos (m_{n-2} (\theta_1-\phi)),$$
$$(g(\phi), C_M(\theta_2,\ldots,\theta_{n-1})   \sin (m_{n-2} \theta_1))= C_M(\theta_2,\ldots,\theta_{n-1}) \sin (m_{n-2} (\theta_1-\phi)),$$
$\span_{\bR} \{ C_M(\theta_2,\ldots,\theta_{n-1}) \cos (m_{n-2} \theta_1), C_M(\theta_2,\ldots,\theta_{n-1})   \sin (m_{n-2} \theta_1)\}$ is a two-dimensional representation of  $\SO(2)$  equivalent to the representation $\bR[1,m_{n-2}]$ with $0<m_{n-2} \leq m.$ If $m_{n-2}=0$, then $\span_{\bR} \{ C_M(\theta_2,\ldots,\theta_{n-1}) \cos (m_{n-2} \theta_1), C_M(\theta_2,\ldots,\theta_{n-1})   \sin (m_{n-2} \theta_1)\}=\bR[1,0]$ is a one-dimensional trivial representation of $\SO(2).$  Moreover,   there are numbers $k_0, \ldots, k_{m-1} \geq 0$ such that  $$\cH^n_m \approx _{\SO(2)} \bR[k_0,0] \oplus \bR[k_1,1] \oplus \ldots \oplus \bR[k_{m-1},m-1] \oplus \bR[1,m].$$
Moreover, $\cH^2_m \approx_{\SO(2)}  \bR[1,m], m \geq 0$.
\end{Corollary}

\begin{Corollary}\label{Hnm}
Combining formula \eqref{stopienId} with Corollary \ref{vilen} we obtain for $n \geq 3$, $m>0$ the following $$\nabla_{\SO(2)}\text{-deg}(-Id,B(\cH^n_m))=(\alpha_0, \alpha_1, \ldots, \alpha_i,\ldots) \in U(\SO(2)),$$ where $\alpha_i=\left\{\begin{array}{lcl}  (-1)^{k_0} & \textrm{ if } & i=0, \\ (-1)^{k_0+1}& \textrm{ if } & i=m, \\ (-1)^{k_0+1} k_i& \textrm{ if } & i=1, \ldots,m-1, \\ 0 &\textrm{ if } & i>m. \end{array} \right.$
\\ Moreover,  $\nabla_{\SO(2)}\text{-deg}(-Id,B(\cH^2_m))=$ $(\alpha_0, \alpha_1, \ldots, \alpha_i,\ldots) \in U(\SO(2)),$ where, for $m>0$,\\ $\alpha_i=\left\{\begin{array}{rcl}  1 & \textrm{ if } & i=0, \\ -1& \textrm{ if } & i=m, \\ 0& \textrm{ if } & i\neq 0,m. \end{array} \right.$
\end{Corollary}

\begin{Remark} \label{deghnm} Suppose that $n \geq  2$ and $m \geq 0.$ Then from the above corollary it follows that if $$\nabla_{\SO(2)}\text{-deg}(-Id,B(\cH^n_m))=(\alpha_0,\alpha_1,\ldots, \alpha_i, \ldots) \in U(\SO(2)),$$
then $\alpha_m=(-1)^{\dim \cH^n_m+1}$ and $\alpha_i=0$ for every $i>m.$
\end{Remark}

To illustrate the above lemma we consider the following examples.

\begin{Example}
Suppose that $n=2$ and $m \geq 0.$ Then $\cH^2_m=\span_{\bR} \{\cos m \phi, \sin m \phi\}$ and $\cH^2_m \approx \bR[1,m],$ where action of $\SO(2)$ ($\approx S^1$) is given by shift in time.
\end{Example}

\begin{Example}
Suppose that $n=3$, $m \geq 0.$ Then $\cH^3_m$ is equivalent to a representation of  $\SO(2)$ of the form $\bR[1,0] \oplus \bR[1,1]  \oplus \ldots \oplus \bR[1,m].$
\end{Example}

Define $\displaystyle V^-_{m_0} \oplus V^0_{m_0} \oplus V^+_{m_0} := \bigoplus_{i=1}^{m_0-1} V^{\gamma}_{-\Delta_{S^n}}(\lambda_i) \oplus V^{\gamma}_{-\Delta_{S^n}}(\lambda_{m_0}) \oplus  \bigoplus_{i=m_0)+1}^{\infty} V^{\gamma}_{-\Delta_{S^n}}(\lambda_i).$  Set $\mu_{m_0}=\dim V^{\gamma}_{-\Delta_{S^{n}}}(\lambda_{m_0})$ and $\nu_{m_0}=\mu_1+\ldots+\mu_{m_0}$  for $m_0 \in \bN$.

In the two lemmas below we present formulas for bifurcation indexes and their properties. We prove only Lemma \ref{pm}. The proof of Lemma \ref{pp} is in spirit the same as the proof of Lemma \ref{pm}.

\begin{Lemma}\label{pm}
Assume that $p_->0$ and fix  $\lambda_{m_0}\in\sigma(-\Delta_{S^n};B(\gamma))$. Then
$$
\BIF_{\SO(n)}(\lambda_{m_0})=$$
\begin{equation} \label{indexpm} \nabla_{\SO(n)}\text{-}\deg(-\Id,B(V_{m_0}^-))^{p_-} \ast
\left((\nabla_{\SO(n)}\text{-}\deg(-\Id,B(V^{\gamma}_{-\Delta_{S^n}}(\lambda_{m_0}))))^{p_-} -\bI\right) \in U(\SO(n)).
\end{equation}
Moreover,
\begin{enumerate}
\item \label{pm1} if $\BIF_{\SO(2)}(\lambda_{m_0})=
i^{\star}(\BIF_{\SO(n)}(\lambda_{m_0}))=(\alpha_0,\alpha_1, \ldots, \alpha_k,\ldots)$ then $$\alpha_k=\left\{\begin{array}{rcl}  (-1)^{\nu_{m_0}p_- } - (-1)^{\nu_{m_0-1}p_-}& \text{ if } & k=0,\\ (-1)^{1+\nu_{m_0}p_-} p_-& \text{ if } & k=m_0-1,\\ 0 & \text{ if } & k \geq  m_0,\end{array}  \right.$$

\item\label{pm2} if $p_-$ is even, then
\[\BIF_{\SO(2)}(\lambda_{m_0})=\nabla_{\SO(2)}\text{-}
\deg(-\Id,B(V^{\gamma}_{-\Delta_{S^n}}(\lambda_{m_0})))^{p_-}-\bI\in U_-(\SO(2)),
\]
\item \label{pm3} if $\dim V^{\gamma}_{-\Delta_{S^n}}(\lambda_{m_0})$ and $p_-\cdot \dim V^-_{m_0}$ are even, then
\[\BIF_{\SO(2)}(\lambda_{m_0})=\nabla_{\SO(2)}\text{-}\deg(-\Id,B(V^{\gamma}_{-\Delta_{S^n}}(\lambda_{m_0})))^{p_-}-\bI\in U_-(\SO(2)),
\]

\item \label{pm4} if $\dim V^{\gamma}_{-\Delta_{S^n}}(\lambda_{m_0})$ is even and $p_-\cdot \dim V^-_{m_0}$ is odd, then
\[\BIF_{\SO(2)}(\lambda_{m_0})=\bI-\nabla_{\SO(2)}\text{-}\deg(-\Id,B(V^{\gamma}_{-\Delta_{S^n}}(\lambda_{m_0})))^{p_-}\in U_+(\SO(2)).
\]
\end{enumerate}
\end{Lemma}
\begin{proof}
Since $\epsilon >0$ is sufficiently small, $$\nabla^2 \Phi(0,\lambda_{m_0} \pm \epsilon)=(- \alpha_1u_1-(\lambda_{m_0} \pm \epsilon) Tu_1,\ldots,- \alpha_p u_p-(\lambda_{m_0} \pm \epsilon) Tu_p)$$ is an isomorphism (a product of isomorphisms) and that is why by the Cartesian product formula of the degree we obtain
$$\BIF_{\SO(n)}(\lambda_{m_0})=$$
$$\begin{array}{lcl} & = &
\nabla_{\SO(n)}\text{-}\deg(\nabla^2_u\Phi(0,\lambda_{m_0}+\epsilon),B(\bH)) -\nabla_{\SO(n)}\text{-}\deg(\nabla^2_u\Phi(0,\lambda_{m_0}-\epsilon), B(\bH)) \\
& = & \displaystyle  \prod_{i=1}^p \nabla_{\SO(n)}\text{-}\deg (-\alpha_i \text{Id} -(\lambda_{m_0} + \epsilon)T,B(H^1_0(B(\gamma))))\\
 & - & \displaystyle \prod_{i=1}^p \nabla_{\SO(n)}\text{-}\deg (-\alpha_i \text{Id}-(\lambda_{m_0} - \epsilon)T,B(H^1_0(B(\gamma)))).
\end{array}
$$
Since $\nabla_{\SO(n)}\text{-}\deg(-\text{Id},B(\bH^1_0(B(\gamma))))=\bI \in U(\SO(n)),$ see \cite{degree}, we obtain
$$\begin{array}{rcl} \BIF_{\SO(n)}(\lambda_{m_0})
& = & \displaystyle  \prod_{\alpha_i=-1} \nabla_{\SO(n)}\text{-}\deg ( \text{Id} -(\lambda_{m_0} + \epsilon)T,B(H^1_0(B(\gamma)))) \\
 & - & \displaystyle \prod_{\alpha_i=-1}^p \nabla_{\SO(n)}\text{-}\deg (\text{Id}-(\lambda_{m_0} - \epsilon)T,B(H^1_0(B(\gamma)))) \\
 & = &  \left(\nabla_{\SO(n)}\text{-}\deg ( \text{Id} -(\lambda_{m_0} + \epsilon)T,B(H^1_0(B(\gamma))))\right)^{p_-}  \\
   &  - &  \left(\nabla_{\SO(n)}\text{-}\deg ( \text{Id} -(\lambda_{m_0} - \epsilon)T,B(H^1_0(B(\gamma))))\right)^{p_-}.
\end{array}
$$
Note that
$$\nabla_{\SO(n)}\text{-}\deg ( \text{Id} -(\lambda_{m_0} + \epsilon)T,B(H^1_0(B(\gamma))))=
$$ $$ =\nabla_{\SO(n)}\text{-}\deg ( -\text{Id} ,B(V^-_m))  \ast \nabla_{\SO(n)}\text{-}\deg ( -\text{Id} ,B(V^0_m))  \ast \nabla_{\SO(n)}\text{-}\deg ( \text{Id} ,B(V^+_m))$$ and
$$\nabla_{\SO(n)}\text{-}\deg ( \text{Id} -(\lambda_{m_0} - \epsilon)T,B(H^1_0(B(\gamma))))$$ $$ =\nabla_{\SO(n)}\text{-}\deg ( -\text{Id} ,B(V^-_m))  \ast \nabla_{\SO(n)}\text{-}\deg ( \text{Id} ,B(V^0_m))  \ast \nabla_{\SO(n)}\text{-}\deg (\text{Id} ,B(V^+_m)).$$
Since for any representation $W$ of  $\SO(n),$ $\nabla_{\SO(n)}\text{-}\deg ( \text{Id} ,B(W)) = \bI \in U(\SO(n) $, we obtain
$$\BIF_{\SO(n)}(\lambda_0) $$ $$ =\nabla_{\SO(n)}\text{-}\deg ( -\text{Id} ,B(V^-_m))^{p_-}  \ast \nabla_{\SO(n)}\text{-}\deg ( -\text{Id} ,B(V^0_m))^{p_-} - \nabla_{\SO(n)}\text{-}\deg ( -\text{Id} ,B(V^-_m))^{p_-}  $$
$$= \nabla_{\SO(n)}\text{-}\deg ( -\text{Id} ,B(V^-_m))^{p_-} \ast \left(  \nabla_{\SO(n)}\text{-}\deg ( -\text{Id} ,B(V^0_m))^{p_-} - \bI\right),$$ which completes the proof of the first part of this lemma.

\noindent \eqref{pm1} Taking into account formula \eqref{indexpm} we obtain $$
\BIF_{\SO(2)}(\lambda_{m_0})= i^{\star}(\BIF_{\SO(n)}(\lambda_{m_0}))$$
\begin{equation} \label{so2n}  =\nabla_{\SO(2)}\text{-}\deg(-\Id,B(V_{m_0}^-))^{p_-} \ast
\left((\nabla_{\SO(2)}\text{-}\deg(-\Id,B(V^{\gamma}_{-\Delta_{S^n}} (\lambda_{m_0}))))^{p_-} -\bI\right) \in U(\SO(2)).
\end{equation}
By the Cartesian product formula for the degree for $\SO(2)$-equivariant gradient maps we have $\nabla_{\SO(2)}\text{-}\deg(-\Id,B(V_{m_0}^-))^{p_-}= \nabla_{\SO(2)}\text{-}\deg(-\Id,B(V_{m_0}^- \times \ldots \times V_{m_0}^-))=\beta=(\beta_0,\beta_1,\ldots,\beta_k,\ldots).$ Combining Corollaries \ref{55}, \ref{Hnm} with Remark \ref{deghnm} we obtain
$$\beta_k=\left\{\begin{array}{rcl}  (-1)^{\nu_{m_0-1}p_-} & \text{ if } & k=0,\\ 0 & \text{ if } & k \geq  m_0-1\end{array} \right.$$
and
\\ $\nabla_{\SO(2)}\text{-}\deg(-\Id,B(V^{\gamma}_{-\Delta_{S^n}} (\lambda_{m_0}))))^{p_-} = \nabla_{\SO(2)}\text{-}\deg(-\Id,B(V^{\gamma}_{-\Delta_{S^n}} (\lambda_{m_0}) \times \ldots \times V^{\gamma}_{-\Delta_{S^n}} (\lambda_{m_0}))) = $
$=\gamma=(\gamma_0,\gamma_1,\ldots,\gamma_k,\ldots).$ Once more, applying Corollaries \ref{55}, \ref{Hnm} with Remark \ref{deghnm}, we obtain
$$\gamma_k=\left\{\begin{array}{rcl}  (-1)^{\mu_{m_0}p_-} & \text{ if } & k=0,\\
(-1)^{1+\mu_{m_0}p_-} p_-& \text{ if } & k=m_0-1, \\
 0 & \text{ if } & k \geq  m_0.\end{array} \right.$$
Now formula \eqref{so2n} has the following form
$$\BIF_{\SO(2)}(\lambda_{m_0})=\beta \ast (\gamma - \bI) = \alpha = (\alpha_0,\alpha_1, \ldots, \alpha_k,\ldots).$$
By  \eqref{mult} we obtain
$$\alpha_k=\left\{\begin{array}{rcl}  (-1)^{1+\nu_{m_0}p_-} p_-& \text{ if } & k=m_0-1,\\ 0 & \text{ if } & k \geq  m_0,\end{array} \right.$$ which completes the proof.

\noindent \eqref{pm2}  Denote
$$\begin{array}{rcl}  \nabla_{\SO(n)}\text{-}\deg(-\Id,B(V_{m_0}^-))^{p_-} & = & (\alpha_0, \alpha_1,\ldots, \alpha_k,\ldots), \\ \nabla_{\SO(n)}\text{-}\deg(-\Id,B(V^{\gamma}_{-\Delta_{S^n}}(\lambda_{m_0}))))^{p_-} -\bI & = & (\beta_0, \beta_1,\ldots, \beta_k,\ldots).\end{array}  $$
Since $p_- > 0$ is even, combining formulas \eqref{mult}, \eqref{stopienId} we obtain that $\alpha_0=1,\beta_0=0$ and $\beta_k \leq 0,$ for $ k \geq 1.$ Finally by  \eqref{mult} we obtain
$$(1,\alpha_1,\ldots, \alpha_k, \ldots) \ast (0,\beta_1,\ldots,\beta_k,\ldots)= (0,\beta_1,\ldots,\beta_k,\ldots) \in U_-(\SO(2)),$$ which completes the proof.

\noindent \eqref{pm3}  Denote
$$\begin{array}{rcl}  \nabla_{\SO(n)}\text{-}\deg(-\Id,B(V_{m_0}^-))^{p_-} & = & (\alpha_0, \alpha_1,\ldots, \alpha_k,\ldots), \\ \nabla_{\SO(n)}\text{-}\deg(-\Id,B(V^{\gamma}_{-\Delta_{S^n}}(\lambda_{m_0}))))^{p_-} -\bI & = & (\beta_0, \beta_1,\ldots, \beta_k,\ldots).\end{array}  $$
Since $p_- \cdot \dim V^-_{m_0}$ and $\dim V^{\gamma}_{-\Delta_{S^n}}(\lambda_{m_0})$ are  even, combining formulas \eqref{mult}, \eqref{stopienId} we obtain that $\alpha_0=1$, $\beta_0=0$ and $\beta_k \leq 0,$ for $ k \geq 1.$ Finally by formula \eqref{mult} we obtain
$$(1,\alpha_1,\ldots, \alpha_k, \ldots) \ast (0,\beta_1,\ldots,\beta_k,\ldots)= (0,\beta_1,\ldots,\beta_k,\ldots) \in U_-(\SO(2)),$$ which completes the proof.

\noindent \eqref{pm4} Denote
$$\begin{array}{rcl}  \nabla_{\SO(n)}\text{-}\deg(-\Id,B(V_{m_0}^-))^{p_-} & = & (\alpha_0, \alpha_1,\ldots, \alpha_k,\ldots), \\ \nabla_{\SO(n)}\text{-}\deg(-\Id,B(V^{\gamma}_{-\Delta_{S^n}}(\lambda_{m_0}))))^{p_-} -\bI & = & (\beta_0, \beta_1,\ldots, \beta_k,\ldots).\end{array}  $$
Since $p_- \cdot \dim V^-_{m_0}$ is odd and $\dim V^{\gamma}_{-\Delta_{S^n}}(\lambda_{m_0})$ is  even, combining formulas \eqref{mult}, \eqref{stopienId} we obtain that $\alpha_0=-1$, $\beta_0=0$ and $\beta_k \leq 0,$ for $ k \geq 1.$ Finally, by formula \eqref{mult} we obtain
$$(-1,\alpha_1,\ldots, \alpha_k, \ldots) \ast (0,\beta_1,\ldots, \beta_k, \ldots)= (0,-\beta_1,\ldots,-\beta_k,\ldots) \in U_+(\SO(2)),$$ which completes the proof.
\end{proof}

In the following lemma we consider the case $p_+>0$.

\begin{Lemma}\label{pp}
Assume that $p_+>0$ and fix  $\lambda_{m_0}\in\sigma(-\Delta_{S^{n}};B(\gamma))$. Then
$$\BIF_{\SO(n)}(-\lambda_{m_0})= $$
$$=\left(\nabla_{\SO(n)}\text{-}\deg(-\Id,B(V^-_{m_0}))\right)^{-p_+}
 \ast
\left((\nabla_{\SO(n)}\text{-} \deg(-\Id,B(V^{\gamma}_{-\Delta_{S^n}}(\lambda_{m_0}))))^{p_+} -\bI\right).
$$
Moreover,
\begin{enumerate}
\item  \label{pp1} if $\BIF_{\SO(2)}(-\lambda_{m_0})=
i^{\star}(\BIF_{\SO(n)}(-\lambda_{m_0}))=(\alpha_0,\alpha_1, \ldots, \alpha_k,\ldots)$ then $$\alpha_k=\left\{\begin{array}{rcl}
(-1)^{\nu_{m_0} p_+} - (-1)^{\nu_{m_0-1}p_+} & \text{ if } & k=0, \\
(-1)^{1+\nu_{m_0}p_+} p_+& \text{ if } & k=m_0-1,\\
0 & \text{ if } & k \geq  m_0\end{array},  \right.$$

\item \label{pp2}  if $p_+$ is even, then
\[\BIF_{\SO(2)}(\lambda_{m_0})=\nabla_{\SO(2)}\text{-}\deg(-\Id,B(V^{\gamma}_{-\Delta_{S^n}}(\lambda_{m_0}))))^{p_+}-\bI\in U_-(\SO(2)),
\]

\item \label{pp3} if $\dim V_{-\Delta_{S^n}}(\lambda_{m_0})$ and $p_+\cdot \dim V^-_{m_0}$ are even, then
\[\BIF_{\SO(2)}(\lambda_{m_0})=\nabla_{\SO(2)}\text{-}\deg(-\Id,B(V^{\gamma}_{-\Delta_{S^n}}(\lambda_{m_0})))^{p_+}-\bI\in U_-(\SO(2)),
\]

\item \label{pp4} if $\dim V_{-\Delta_{S^n}}(\lambda_{m_0})$ is even and $p_+\cdot \dim V^-_{m_0}$ is odd, then
\[\BIF_{\SO(2)}(\lambda_{m_0})=\bI-\nabla_{\SO(2)}\text{-}\deg(-\Id,B(V^{\gamma}_{-\Delta_{S^n}}(\lambda_{m_0})))^{p_+}\in U_+(\SO(2)).
\]
\end{enumerate}
\end{Lemma}

\end{document}